\documentclass[a4paper,11pt]{amsart}

\usepackage{latexsym}
\usepackage{amssymb,amsmath}
\usepackage{graphics}
\usepackage{url}
\usepackage{tikz}
\usetikzlibrary{patterns}
\usepackage{mathdots}

\usepackage[vcentermath]{youngtab}
  
\usepackage{booktabs}  

\newtheorem{theorem}{Theorem}
\newtheorem*{theorem*}{Theorem}

\newtheorem{proposition}[theorem]{Proposition}

\newtheorem{lemma}[theorem]{Lemma}

\theoremstyle{definition}
\newtheorem*{definition}{Definition}

\newcommand{\N}{\mathbf{N}}

\newcommand{\R}{\mathbf{R}}
\newcommand{\CC}{\mathbf{C}}

\renewcommand{\epsilon}{\varepsilon}

\renewcommand{\theta}{\vartheta}

\newcommand{\D}{D}

\newcounter{thmlistcnt}
\newenvironment{thmlist}%
	{\setcounter{thmlistcnt}{0}%
	\begin{list}{\emph{(\roman{thmlistcnt})}}{%
		\usecounter{thmlistcnt}%
		\setlength{\topsep}{0pt}%
		\setlength{\leftmargin}{0pt}%
		\setlength{\itemsep}{0pt}%
		\setlength{\labelwidth}{17pt}
		\setlength{\itemindent}{30pt}}%
	}%
	{\end{list}}%
	
\makeatletter 
\newcommand\mynobreakpar{\par\nobreak\@afterheading} 
\makeatother

\newlength{\hatchspread}
\newlength{\hatchthickness}
\newlength{\hatchshift}
\newcommand{\hatchcolor}{}
\tikzset{hatchspread/.code={\setlength{\hatchspread}{#1}},
         hatchthickness/.code={\setlength{\hatchthickness}{#1}},
         hatchshift/.code={\setlength{\hatchshift}{#1}},% must be >= 0
         hatchcolor/.code={\renewcommand{\hatchcolor}{#1}}}
\tikzset{hatchspread=3pt,
         hatchthickness=0.4pt,
         hatchshift=0pt,% must be >= 0
         hatchcolor=black}
\pgfdeclarepatternformonly[\hatchspread,\hatchthickness,\hatchshift,\hatchcolor]% variables
   {custom north east lines}% name
   {\pgfqpoint{\dimexpr-2\hatchthickness}{\dimexpr-2\hatchthickness}}% lower left corner
   {\pgfqpoint{\dimexpr\hatchspread+2\hatchthickness}{\dimexpr\hatchspread+2\hatchthickness}}% upper right corner
   {\pgfqpoint{\dimexpr\hatchspread}{\dimexpr\hatchspread}}% tile size
   {% shape description
    \pgfsetlinewidth{\hatchthickness}
    \pgfpathmoveto{\pgfqpoint{\dimexpr\hatchshift-0.15pt}{-0.15pt}}
    \pgfpathlineto{\pgfqpoint{\dimexpr\hatchspread+0.15pt}{\dimexpr\hatchspread-\hatchshift+0.15pt}}
    \ifdim \hatchshift > 0pt
      \pgfpathmoveto{\pgfqpoint{-0.15pt}{\dimexpr\hatchspread-\hatchshift-0.15pt}}
      \pgfpathlineto{\pgfqpoint{\dimexpr\hatchshift+0.15pt}{\dimexpr\hatchspread+0.15pt}}
    \fi
    \pgfsetstrokecolor{\hatchcolor}
    \pgfusepath{stroke}
   }	
	
\definecolor{grey}{rgb}{0.8,0.8,0.8}	
\newcommand{\youngbox}[2]{\draw  (#2,-#1) rectangle (#2-1,-#1-1);}
\newcommand{\fillbox}[3]{\draw[fill=#3] (#2,-#1) rectangle (#2-1,-#1-1);}
\newcommand{\hatchbox}[3]{\draw[pattern=custom north east lines, hatchspread=#3] (#2,-#1) rectangle (#2-1,-#1-1);}
\newcommand{\entry}[3]{\node at (#2-0.5,-#1-0.5) {#3};}
\newcommand{\entryg}[3]{\node at (#2-0.5,-#1-0.5) {#3};}
\newcommand{\entryb}[3]{\node at (#2-0.5,-#1-0.5) {\textbf{#3}};}

\newcommand{\bl}{(\lambda)}
\begin{document}
\title[]{}
\date{\today}

\let\Ssymbol\S
\renewcommand{\R}{\mathcal{R}}
\newcommand{\C}{\mathcal{C}}
\renewcommand{\H}{\mathcal{H}}
\newcommand{\Rpp}{R^{\prime +}}
\newcommand{\Cpp}{C^{\prime +}}
\newcommand{\NW}{{\mathrm{NW}}}
\newcommand{\SE}{{\mathrm{SE}}}
\newcommand{\NE}{{\mathrm{NE}}}
\newcommand{\SW}{{\mathrm{SW}}}
\newcommand{\SSYT}{\mathrm{SSYT}}

\begin{abstract}
We use the hook lengths of a partition to define
two rectangular tableaux. We prove these tableaux have equal multisets of entries,
first by elementary combinatorial arguments, and
then using Stanley's Hook Content Formula and symmetric polynomials.
\end{abstract}

\subjclass[2010]{Primary 05E05, Secondary: 05E10}

\title[A corollary of Stanley's Hook Content Formula]{A 
corollary of Stanley's Hook Content Formula}
\author{Mark Wildon}
\date{October 2018}
\maketitle

\section{Introduction}

This paper presents two proofs of an appealing 
corollary of Stanley's Hook Content Formula \cite[Theorem 7.21.2]{StanleyII}
 for the number of semistandard Young
tableaux: the first is self-contained and entirely elementary, 
while the  second uses Stanley's result and symmetric polynomials.
The author hopes the paper will be useful as an introduction to 
 this interesting 
circle of ideas.
\thispagestyle{empty}

%It is shown in \cite[Corollary 7.21.6]{StanleyII} 
%that Stanley's formula implies the Hook Formula for the number of standard Young tableaux,
%which takes pride of place in the impressive list of combinatorial formulae
%collected on MathOverflow \cite{MathOverflowHookFormula}. 
The following definitions are standard. %We require the following standard definitions: 
A \emph{partition} of $n \in \N_0$ is a sequence $(\lambda_1,\ldots, \lambda_k)$
of natural numbers
such that $\lambda_1 \ge \ldots \ge \lambda_k$ and $\lambda_1 + \cdots + \lambda_k = n$. The \emph{size} of $\lambda$ is $n$.
%The $\lambda_i$ are called the \emph{parts} of $\lambda$. 
We define $\ell(\lambda) = k$ and $a(\lambda) = \lambda_1$, setting $\ell(\varnothing) = a(\varnothing) = 0$.
The \emph{Young diagram} of $\lambda$, denoted $[\lambda]$, 
is the set of \emph{boxes}
$\{ (i,j) : 1 \le i \le \ell(\lambda), 1 \le j \le \lambda_i \}$.

We fix throughout $r$, $c \in \N$.
%. and work with partitions $\lambda$ 
%with $\ell(\lambda) \le r$ and $a(\lambda) \le s$. 
Let $D = \{ (i,j) : 1 \le i \le r, 1 \le j \le c\}$.
We orient $D$ by compass directions, thus $(r,1)$ is the box in its south-west corner
and $(1,c)$ is the box in its north-east corner.
As a running example, 
the Young diagram of $(7,5,4,3,3,2)$, shown as a subset of $D$
when $r = 6$ and $c= 8$, is below.

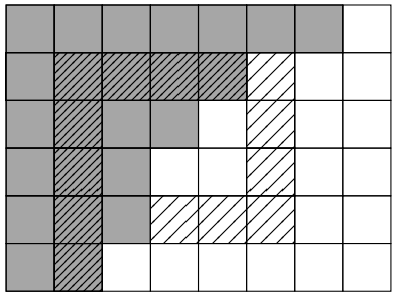
\begin{center}
\begin{tikzpicture}[x=18pt, y=18pt]
\definecolor{darkgrey}{rgb}{0.65,0.65,0.65}
\foreach \j in {1,2,3,4,5,6,7}
	\fillbox{1}{\j}{darkgrey}
\foreach \j in {1,2,3,4,5}
	\fillbox{2}{\j}{darkgrey};
\foreach \j in {1,2,3,4}
	\fillbox{3}{\j}{darkgrey};
\foreach \j in {1,2,3}
	\fillbox{4}{\j}{darkgrey};
\foreach \j in {1,2,3}
	\fillbox{5}{\j}{darkgrey};
\foreach \j in {1,2}
	\fillbox{6}{\j}{darkgrey};
\hatchbox{2}{2}{3pt}
\hatchbox{2}{3}{3pt}\hatchbox{2}{4}{3pt}\hatchbox{2}{5}{3pt}
\hatchbox{3}{2}{3pt}\hatchbox{4}{2}{3pt}\hatchbox{5}{2}{3pt}\hatchbox{6}{2}{3pt}		
\hatchbox{5}{6}{6pt}
\hatchbox{4}{6}{6pt}\hatchbox{3}{6}{6pt}\hatchbox{2}{6}{6pt}
\hatchbox{5}{5}{6pt}\hatchbox{5}{4}{6pt}
%\hatchbox{4}{6}{3pt}
%\hatchbox{3}{6}{3pt}\hatchbox{2}{6}{3pt}
%\hatchbox{4}{5}{3pt}\hatchbox{4}{4}{3pt}
\foreach \i in {1,2,3,4,5,6}
	\foreach \j in {1,2,3,4,5,6,7,8}
		\youngbox{\i}{\j};
\end{tikzpicture}
\end{center}
\noindent 
The hatched squares show 
the hooks on $(2,2)$ and $(5,6)$, as %; they are of hook lengths $8$ and $6$ respectively, as
defined formally in the definition below.

%The hatched squares show the hook on $(2,2) \in [\lambda]$, consisting of all boxes in $[\lambda]$
%that are weakly east or south of $(2,2)$, and the hook on $(5,6) \not\in [\lambda]$, consisting
%of all boxes not in $[\lambda]$ that are weakly west or north of $(5,6)$.
%below show the Young diagram of the partition $(7,5,4,4,3,2,2)$ as a subset of the Young
%diagram of the partition $(8,8,8,8,8,8,8)$. 
%The shaded squares are the hooks on $(2,2)$ and $(5,6)$, as defined below.

%We need two further less standard definitions.

\begin{definition}
Let $\lambda$ be a partition with $\ell(\lambda) \le r$ and $a(\lambda) \le c$.
Let $(i,j) \in D$. 
\begin{itemize}
\item[(i)] 
The \emph{hook} on $(i,j)$, denoted $H_{(i,j)}\bl$, is
\begin{align*}
&\qquad\quad \{ (i,j) \} \cup \bigl\{ (i',j) \in [\lambda] : i' > i \bigr\} \cup \bigl\{ (i, j') \in [\lambda] : j' > j \bigr\}\\
\intertext{if $(i,j) \in [\lambda]$ and}
&\qquad\quad \{ (i,j) \} \cup \bigl\{ (i', j) \in D \backslash [\lambda] : i' < i\bigr\} \cup \bigl\{ (i,j') \in D \backslash
[\lambda] : j' < j \bigr\} \end{align*}
if $(i,j) \in D\backslash [\lambda]$.
We define the
\emph{hook length} of $(i,j)$, denoted $h_{(i,j)}\bl$, to be~$|H_{(i,j)}(\lambda)|$.
\item[(ii)] The \emph{distance} of $(i,j)$, denoted $d_{(i,j)}\bl$, is the number of boxes
in any walk by steps south and west to $(r,1)$ if $(i,j) \in [\lambda]$,
or of any walk by steps north and east to $(1,c)$ if $(i,j) \in \D \backslash [\lambda]$.
\end{itemize}
\end{definition}

Our result concerns two ways to fill the boxes of $D$ with natural numbers.
Formally, these are specified by two functions from $D$ to $\N$, assigning
to each box of $D$ a corresponding \emph{entry} in $\N$.

\begin{definition}
Let $\lambda$ be a partition with $\ell(\lambda) \le r$ and $a(\lambda) \le c$.
Let $(i,j) \in D$.
The \emph{hook/distance tableau for $\lambda$} has entry in box $(i,j)$
%either $h_{(i,j)}(\lambda)$, if $(i,j) \in [\lambda]$, or $d_{(i,j)}(\lambda)$, if
%$(i,j) \in D\backslash \lambda$. The \emph{distance/hook tableau for $\lambda$} has
%entry in box $(i,j)$ either $d_{(i,j)}(\lambda)$, if $(i,j) \in [\lambda]$, or
%$h_{(i,j)}(\lambda)$, if $(i,j) \in D\backslash \lambda$.
%is the function $D \rightarrow \N$ defined
%by
\begin{align*} 
%(i,j) \mapsto 
\begin{cases} h_{(i,j)}\bl & \text{ if $(i,j) \in [\lambda]$} \\
d_{(i,j)}\bl & \text{ if $(i,j) \in D\backslash [\lambda]$.} \end{cases}
\intertext{The \emph{distance/hook tableau for $\lambda$} has entry in box $(i,j)$}
%is the function $D \rightarrow \N$ defined by
\begin{cases} d_{(i,j)}\bl & \text{ if $(i,j) \in [\lambda]$} \\
h_{(i,j)}\bl & \text{ if $(i,j) \in D\backslash [\lambda]$.} \end{cases} 
\end{align*}
\end{definition}

\begin{theorem}\label{thm:main} 
For any partition $\lambda$ with $\ell(\lambda) \le r$ and $a(\lambda) \le c$,
the multisets of entries of the hook/distance tableau for $\lambda$ 
and the distance/hook tableau for~$\lambda$ are equal.
\end{theorem}

In our running example, the hook/distance tableau (below left) and distance/hook tableau (below right) 
both
have, for instance, six entries of $1$, three entries of $8$, and $12$ as their unique greatest entry.

\medskip
\begin{center}
\begin{tikzpicture}[x=18pt, y=18pt]
\foreach \j in {1,2,3,4,5,6,7}
	\fillbox{1}{\j}{grey}
\foreach \j in {1,2,3,4,5}
	\fillbox{2}{\j}{grey};
\foreach \j in {1,2,3,4}
	\fillbox{3}{\j}{grey};
\foreach \j in {1,2,3}
	\fillbox{4}{\j}{grey};
\foreach \j in {1,2,3}
	\fillbox{5}{\j}{grey};
\foreach \j in {1,2}
	\fillbox{6}{\j}{grey};
\foreach \i in {1,2,3,4,5,6}
	\foreach \j in {1,2,3,4,5,6,7,8}
		\youngbox{\i}{\j};
\entry{1}{1}{12}\entry{1}{2}{11}\entry{1}{3}{9}\entry{1}{4}{6}\entry{1}{5}{4}\entry{1}{6}{2}\entry{1}{7}{1}\entry{1}{8}{1}	
\entry{2}{1}{9}\entry{2}{2}{8}\entry{2}{3}{6}\entry{2}{4}{3}\entry{2}{5}{1}\entry{2}{6}{4}\entry{2}{7}{3}\entry{2}{8}{2}	
\entry{3}{1}{7}\entry{3}{2}{6}\entry{3}{3}{4}\entry{3}{4}{1}\entry{3}{5}{6}\entry{3}{6}{5}\entry{3}{7}{4}\entry{3}{8}{3}	
\entry{4}{1}{5}\entry{4}{2}{4}\entry{4}{3}{2}\entry{4}{4}{8}\entry{4}{5}{7}\entry{4}{6}{6}\entry{4}{7}{5}\entry{4}{8}{4}	
\entry{5}{1}{4}\entry{5}{2}{3}\entry{5}{3}{1}\entry{5}{4}{9}\entry{5}{5}{8}\entry{5}{6}{7}\entry{5}{7}{6}\entry{5}{8}{5}	
\entry{6}{1}{2}\entry{6}{2}{1}\entry{6}{3}{11}\entry{6}{4}{10}\entry{6}{5}{9}\entry{6}{6}{8}\entry{6}{7}{7}\entry{6}{8}{6}	
\end{tikzpicture}\qquad
\begin{tikzpicture}[x=18pt, y=18pt]
\foreach \j in {1,2,3,4,5,6,7}
	\fillbox{1}{\j}{grey}
\foreach \j in {1,2,3,4,5}
	\fillbox{2}{\j}{grey};
\foreach \j in {1,2,3,4}
	\fillbox{3}{\j}{grey};
\foreach \j in {1,2,3}
	\fillbox{4}{\j}{grey};
\foreach \j in {1,2,3}
	\fillbox{5}{\j}{grey};
\foreach \j in {1,2}
	\fillbox{6}{\j}{grey};
\foreach \i in {1,2,3,4,5,6}
	\foreach \j in {1,2,3,4,5,6,7,8}
		\youngbox{\i}{\j};
\entry{1}{1}{6}\entry{1}{2}{7}\entry{1}{3}{8}\entry{1}{4}{9}\entry{1}{5}{10}\entry{1}{6}{11}\entry{1}{7}{12}\entry{1}{8}{1}	
\entry{2}{1}{5}\entry{2}{2}{6}\entry{2}{3}{7}\entry{2}{4}{8}\entry{2}{5}{9}\entry{2}{6}{1}\entry{2}{7}{2}\entry{2}{8}{4}	
\entry{3}{1}{4}\entry{3}{2}{5}\entry{3}{3}{6}\entry{3}{4}{7}\entry{3}{5}{1}\entry{3}{6}{3}\entry{3}{7}{4}\entry{3}{8}{6}	
\entry{4}{1}{3}\entry{4}{2}{4}\entry{4}{3}{5}\entry{4}{4}{1}\entry{4}{5}{3}\entry{4}{6}{5}\entry{4}{7}{6}\entry{4}{8}{8}	
\entry{5}{1}{2}\entry{5}{2}{3}\entry{5}{3}{4}\entry{5}{4}{2}\entry{5}{5}{4}\entry{5}{6}{6}\entry{5}{7}{7}\entry{5}{8}{9}	
\entry{6}{1}{1}\entry{6}{2}{2}\entry{6}{3}{1}\entry{6}{4}{4}\entry{6}{5}{6}\entry{6}{6}{8}\entry{6}{7}{9}\entry{6}{8}{11}	
\end{tikzpicture}
\end{center}

In \Ssymbol\ref{sec:proofComb} we give an elementary combinatorial
proof of Theorem~\ref{thm:main},
working by induction on the size of $\lambda$. Then in \Ssymbol\ref{sec:proofAlg} we put the theorem in its proper
context by giving a shorter algebraic
proof using Stanley's Hook Content Formula and a bijection due to King \cite[\S 4]{KingSU2Plethysms}.

\section{An elementary combinatorial proof of the main theorem}\label{sec:proofComb}

We work by induction on $n$, the size of $\lambda$. If $n=0$, so
$\lambda = \varnothing$,
then the hook/distance tableau has the distances $i+c-1$,\ldots, $i$ from west
to east in row $i$, while the distance/hook tableau has the hook lengths
$i, \ldots i + c - 1$ from west to east in row $i$. Therefore the multisets of entries agree.

Suppose the theorem holds for the partition $\lambda$ of $n$ where $n < rc$.
Let $(a,b) \in D \backslash [\lambda]$
be a box such that $[\lambda] \cup \{(a,b)\}$ is the Young diagram of a partition,
denoted~$\lambda'$. 
As a visual aid, we define the \emph{hook/hook tableau} to have
entry $h_{(i,j)}(\lambda)$ in box $(i,j)$ if $(i,j) \in [\lambda]$ and
entry $h_{(i,j)}(\lambda')$ in box $(i,j)$ if $(i,j) \in D \backslash [\lambda']$. No entry
is assigned to the exceptional box $(a,b)$. The hook/hook tableau in
our running example with $(a,b) = (2,6)$ is below.

\medskip
\begin{center}
\begin{tikzpicture}[x=18pt, y=18pt]
\foreach \j in {1,2,3,4,5,6,7}
	\fillbox{1}{\j}{grey}
\foreach \j in {1,2,3,4,5}
	\fillbox{2}{\j}{grey};
\foreach \j in {1,2,3,4}
	\fillbox{3}{\j}{grey};
\foreach \j in {1,2,3}
	\fillbox{4}{\j}{grey};
\foreach \j in {1,2,3}
	\fillbox{5}{\j}{grey};
\foreach \j in {1,2}
	\fillbox{6}{\j}{grey};
\foreach \i in {1,2,3,4,5,6}
	\foreach \j in {1,2,3,4,5,6,7,8}
		\youngbox{\i}{\j};
\hatchbox{2}{6}{3pt}		
\entryg{1}{1}{12}\entryg{1}{2}{11}\entryg{1}{3}{9}\entryg{1}{4}{6}\entryg{1}{5}{4}\entryb{1}{6}{2}\entryg{1}{7}{1}\entryg{1}{8}{1}	
\entryb{2}{1}{9}\entryb{2}{2}{8}\entryb{2}{3}{6}\entryb{2}{4}{3}\entryb{2}{5}{1}   \entryb{2}{7}{1}\entryb{2}{8}{3}	
\entryg{3}{1}{7}\entryg{3}{2}{6}\entryg{3}{3}{4}\entryg{3}{4}{1}\entryg{3}{5}{1}\entryb{3}{6}{2}\entryg{3}{7}{4}\entryg{3}{8}{6}	
\entryg{4}{1}{5}\entryg{4}{2}{4}\entryg{4}{3}{2}\entryg{4}{4}{1}\entryg{4}{5}{3}\entryb{4}{6}{4}\entryg{4}{7}{6}\entryg{4}{8}{8}	
\entryg{5}{1}{4}\entryg{5}{2}{3}\entryg{5}{3}{1}\entryg{5}{4}{2}\entryg{5}{5}{4}\entryb{5}{6}{5}\entryg{5}{7}{7}\entryg{5}{8}{9}	
\entryg{6}{1}{2}\entryg{6}{2}{1}\entryg{6}{3}{1}\entryg{6}{4}{4}\entryg{6}{5}{6}\entryb{6}{6}{7}\entryg{6}{7}{9}\entryg{6}{8}{11}	
\end{tikzpicture}
\end{center}

%Clearly if $i\not=a$ and $j\not=b$ then
%$h_{(i,j)}(\lambda) = h_{(i,j)}(\lambda')$. 
The following lemma is used below to
express the hook lengths of $\lambda'$ in terms of those of $\lambda$.
The hook/hook tableau above shows that the sets $\R$, $\R'$, $\C$ and $\C'$ in this lemma
are $\{1,3,6,8,9\}$,
$\{2,4,5,7\}$, $\{2\}$ and $\{1,3\}$.

\begin{lemma}\label{lemma:complementaryHooks}{\ }\mynobreakpar
\begin{thmlist}
\item Let $\R = \{ h_{(a,j)}(\lambda) 
: 1 \le j < b \}$ and let $\R' = \{ h_{(i,b)}(\lambda') : a < i \le r \}$.
Then $\R \cup \R' = \{1, \ldots r-a+b-1\}$ where the union is disjoint.

\item Let $\C = \{ h_{(i,b)}(\lambda) : 1 \le i < a \}$
and let $\C' = \{ h_{(a,j)}(\lambda') : b < j \le c \}$.
Then $\C \cup \C' = \{1,\ldots, c-b+a-1\}$ where the union is disjoint.

\end{thmlist}
\end{lemma}

\begin{proof} %[Proof of Lemma~\ref{lemma:complementaryHooks}]
It is clear from the diagram below and the hook/hook tableau
that no hook length in $\R$ can equal a hook length in $\R'$.
%each hook length in $\R$ lies strictly
%between two blocks of consecutive hook lengths in $\R'$.

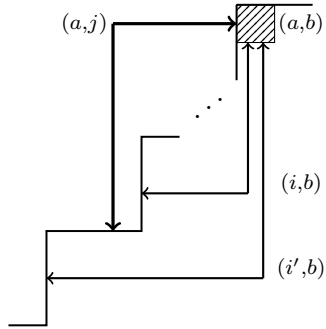
\begin{center}
\begin{tikzpicture}[x=0.5cm,y=0.5cm]

\renewcommand{\a}{6.5}
\renewcommand{\b}{-5.5}
\newcommand{\as}{2.5}
\newcommand{\ass}{1}
\newcommand{\bs}{2.5}
\newcommand{\bss}{1}

\draw[line width = 0.75pt] (8,-2)--(10,-2);
\draw[line width = 0.75pt] (8,-2)--(8,-4);
\node at (7.25,-4.5) {$\iddots$};
\draw[pattern=custom north east lines, hatchspread=3pt] (8,-2) rectangle (9,-3);
\draw[line width = 0.75pt] (\a,\b)--(\a-\ass,\b)--(\a-\ass,\b-\bs)--(\a-\ass-\as,\b-\bs)--(\a-\ass-\as,\b-\bs-\bs)--(\a-\ass-\as-\ass,\b-\bs-\bs); %--(\a-\ass-\as-\as,\b-\bs-\bs-\bss);
\draw[very thick,->] (4.75,-2.5)--(8,-2.5);
\draw[very thick,->] (4.75,-2.5)--(4.75,\b-\bs);
\draw[thick,->] (8.3,-7)--(8.3,-3);
\draw[thick,->] (8.3,-7)--(\a-\ass,-7);
\draw[thick,->] (8.7,-9.25)--(8.7,-3);
\draw[thick,->] (8.7,-9.25)--(\a-\ass-\as,-9.25);
%\draw (4.25,-2) rectangle (5.25,-3);
%\node at (5,-1.9) {$\R$};
%\node at (9.2,-9) {$\S$};
\node at (9.7,-2.5) {$\scriptstyle (a,b)$};
\node at (4,-2.5) {$\scriptstyle (a,j)$};
\node at (9.7,-6.75) {$\scriptstyle (i,b)$};
\node at (9.7,-9) {$\scriptstyle (i',b)$};

\end{tikzpicture}
\end{center}

\noindent Therefore $\R$ and $\R'$ are disjoint.
If $\ell(\lambda) = r$ then the greatest hook length in $\R \cup \R'$ is
$h_{(a,1)} = (r-a)+(b-1) \in \R$, measured by walking north from $(r,1)$ to $(a,1)$ then east to $(a,b-1)$.
Otherwise it is $h_{(r,b)} = b + (r-a-1) \in \R'$, measured by walking
east from $(r,1)$ to $(r,b)$ then north to $(a+1,b)$. Since $|\R| + |\R'| = (b-1) + (r-a)$,
this proves (i); the proof of (ii) is entirely analogous.
\end{proof}

%As a useful notation, 
Given a multiset $\mathcal{X}$ of natural numbers, let $\mathcal{X}^+ = \{x+1 : x \in \mathcal{X} \}$. 
Let~$\cup$ denote the union of multisets. Thus $\{2,2,3\}^+ = \{3,3,4\}$ and
 $\{1,2\} \cup \{2,2,3\} = \{1,2,2,2,3\}$.
 %,
%and so on.

We are now ready for the inductive step. Define
\begin{alignat*}{3}
\H_\NW &=  \{ h_{(i,j)}(\lambda) : (i,j)  \in [\lambda] \}
&\hspace{0.1in}
\H'_\NW &= \{ h_{(i,j)}(\lambda') : (i,j) \in [\lambda'] \} \\
\H_\SE &= \{ h_{(i,j)}(\lambda) : (i,j) \in D \backslash [\lambda] \} 
&\H'_\SE &= \{ h_{(i,j)}(\lambda') : (i,j) \in D \backslash [\lambda'] \}
\end{alignat*}
%\begin{align*}
%\H_\NW &=  \{ h_{(i,j)}(\lambda) : (i,j)  \in [\lambda] \} \\
%\H_\SE &= \{ h_{(i,j)}(\lambda) : (i,j) \in D \backslash [\lambda] \} \\
%\H'_\NW &= \{ h_{(i,j)}(\lambda') : (i,j) \in [\lambda'] \} \\
%\H'_\SE &= \{ h_{(i,j)}(\lambda') : (i,j) \in D \backslash [\lambda'] \}
%\end{align*}
and let $\mathcal{D}_\NW, \mathcal{D}_\SE, \mathcal{D}'_\NW, \mathcal{D}'_\SE$ be
defined analogously, replacing hook lengths with distances. 
%The analogue
%of Lemma~\ref{lemma:complementaryHooks} for distances is as follows.

%\begin{lemma}
%We have $\mathcal{D}'_\NW = \mathcal{D}_\NW 
%\end{lemma}

If $(i,j) \in [\lambda']$ then
\[ h_{(i,j)}(\lambda') = \begin{cases} h_{(i,j)}(\lambda) & \text{if $i\not= a$ and $j\not= b$} \\
h_{(i,j)}(\lambda)+1 & \text{if $i = a$ or $j = b$ but not both} \\
1 & \text{if $i = a$ and $j=b$}\end{cases}\]
and similarly if $(i,j) \in D\backslash [\lambda']$ then
\[ h_{(i,j)}(\lambda') = \begin{cases} h_{(i,j)}(\lambda) & \text{if $i\not= a$ and $j\not= b$} \\
h_{(i,j)}(\lambda)-1 & \text{if $i = a$ or $j = b$ but not both.} \end{cases}. \]
 The equations for $h_{(i,j)}(\lambda')$ above
show that $\H'_\NW$ is obtained from $\H_\NW$ by 
removing each hook length in $\R \cup \C$ and inserting
each hook length in $\R^+ \cup \C^+ \cup \{1\}$. 
Similarly $\H'_\SE$ is obtained from $\H_\SE$ by
removing each hook length in $\Rpp \cup \Cpp \cup \{1\}$ and
inserting each hook length in $\R' \cup \mathcal \C'$. Therefore, using the multiset union,
\begin{align*}
 \H'_\NW \cup \R \cup \C &= \H_\NW \cup \R^+
\cup \C^+ \cup \{1\} \\
\H'_\SE \cup \Rpp \cup \Cpp \cup \{1\} &= \H_\SE \cup \R'\cup \C'. \end{align*}

We now manipulate these equations so that the 
inductive hypothesis $\H_\NW \cup \mathcal{D}_\SE = \mathcal{D}_\NW \cup \H_\SE$
applies.
Recall from Lemma~\ref{lemma:complementaryHooks} 
that $\R \cup \R' = \{1,\ldots, r-a+b-1\}$ and
$\C \cup \C' = \{1,\ldots, c-b+a-1\}$. Hence, taking
the multiset union of both sides of the two equations above with $\R' \cup \C'$ and
$\R^+ \cup \C^+$, respectively, we get
%$\H'_\NW \cup \{1,\ldots,r-a+b-1\}
%\cup \{1,\ldots, s-b+a-1\} =
%\H_\NW \cup \R^+ \cup \C^+ \cup \S \cup \D \cup \{1\}$
%and $\H'_\SE \cup \{2,\ldots, r-a+b\} \cup \{2,\ldots, s-b+a\} \cup \{1\} 
%= \H_\SE \cup \S \cup \mathcal{D} \cup \R^+ \cup \R^+
%\cup \C^+$.
%
%\begin{align*}
%\H'_\NW \cup \{1,\ldots,r-a+b-1\} &
%\cup \{1,\ldots, s-b+a-1\} \\
%&=
%\H_\NW \cup \R^+ \cup \C^+ \cup \S \cup \D \cup \{1\}\\
%\H'_\SE \cup \{2,\ldots, r-a+b\} \cup \{2,\ldots, s-b+a\} \cup \{1\} \\
% &=
%\H_\SE \cup \S \cup \mathcal{D} \cup \R^+ \cup \R^+
%\cup \C^+. 
%\end{align*}
\begin{align*}
\H'_\NW \cup \{1,\ldots,r-a+b-1\} 
\cup \{1,\ldots, c-b+a-1\} &=
\H_\NW \cup \mathcal{Y} \cup \{1\}\\
\H'_\SE \cup \{2,\ldots, r-a+b\} \cup \{2,\ldots, c-b+a\} \cup \{1\} &=
\H_\SE \cup \mathcal{Y} 
\end{align*}
where $Y =  \R^+ \cup \C^+ \cup \R' \cup \C'$.
Setting $\mathcal{Z} = \{1,\ldots, r-a+b-1\} \cup \{2,\ldots, c-b+a\}$, it follows that
\begin{align}
\H'_\NW \cup \mathcal{Z} \cup \{1\} &=
%\{1,\ldots,r-a+b-1\} \cup \{1,\ldots, s-b+a-1\} &=
\H_\NW \cup \mathcal{Y} \cup \{1, c-b+a\} \label{eq:hook1}\\
\H'_\SE \cup \mathcal{Z} \cup \{r-a+b\} &= \label{eq:hook2}
%\{2,\ldots, r-a+b\} \cup \{2,\ldots, s-b+a\} \cup \{1\} &=
\H_\SE \cup \mathcal{Y}. 
\end{align}
Since  $d_{(a,b)}(\lambda') = r-a+b$ and $d_{(a,b)}(\lambda) = c-b+a$
we have $\mathcal{D}'_\SE \cup \{c-b+a\} = \mathcal{D}_\SE$ and 
$\mathcal{D}'_\NW = \mathcal{D}_\NW \cup \{r-a+b\}$. Taking
the multiset union of both sides of~\eqref{eq:hook1} with
 $\mathcal{D}'_\SE \cup \{c-b+a\}$ 
and~\eqref{eq:hook2} with $\mathcal{D}'_\NW$, we get
\begin{align*}
\H'_\NW \cup \mathcal{D}'_\SE \cup \mathcal{Z} \cup \{1,c-b+a\} &=
\H_\NW \cup \mathcal{D}_\SE \cup \mathcal{Y} \cup \{1, c-b+a\} \\
\H'_\SE \cup \mathcal{D}'_\NW \cup \mathcal{Z} \cup \{r-a+b\} &=
\H_\SE \cup \mathcal{D}_\NW \cup \mathcal{Y} \cup \{r-a+b\}.
\end{align*}
Cancelling the elements of $\{1,c-b+a\}$ and $\{r-a+b\}$, we get
$\H'_\NW \cup \mathcal{D}'_\SE \cup \mathcal{Z} = \H_\NW \cup \mathcal{D}_\SE
\cup \mathcal{Y}$ and 
$\H'_\SE \cup \mathcal{D}'_\NW \cup \mathcal{Z} = \H_\SE \cup \mathcal{D}_\NW 
\cup \mathcal{Y}$.
%\begin{align*}
%\H'_\NW \cup \mathcal{D}'_\SE \cup \mathcal{Z} &= \H_\NW \cup \mathcal{D}_\SE
%\cup \mathcal{Y} \\
%\H'_\SE \cup \mathcal{D}'_\NW \cup \mathcal{Z} &= \H_\SE \cup \mathcal{D}_\NW 
%\cup \mathcal{Y}. \end{align*}
By our inductive hypothesis, the right-hand sides are equal.
Therefore %these two equations imply that %the two equations immediately above imply that 
$\H'_\NW \cup \mathcal{D}'_\SE  = \mathcal{D}'_\NW \cup \H'_\SE$, as required.

\section{A symmetric polynomials proof of the main theorem}\label{sec:proofAlg}

\subsection{Background}\label{subsec:symBackground}
Fix a partition $\lambda$. 
A \emph{$\lambda$-tableau}
 is a function $[\lambda] \rightarrow \N$ assigning
to each box of $[\lambda]$ an entry in $\N$. A $\lambda$-tableau $t$ is \emph{semistandard}
if its rows are weakly increasing, when read from west to east, and its columns are strictly
increasing, when read from north to south.
Let $\SSYT_r(\lambda)$ denote the set of semistandard $\lambda$-tableaux 
with maximum entry at most our fixed number~$r$.
For $t \in \SSYT_r(\lambda)$, let $x^t$ denote the monomial $x_1^{e_1} \ldots x_r^{e_r}$ 
where $e_k$ is the number of entries of $t$ equal to~$k$. By definition,
the \emph{Schur polynomial}~$s_\lambda$ in~$r$ variables is
\[ s_\lambda(x_1,\ldots,x_r) = \sum_{t \in \SSYT_r(\lambda)} x^t. \]
%We regard $s_\lambda$ as the $r$-variable specialization of the symmetric function
%defined, for instance, in \cite[Definition 7.10.1]{StanleyII}.
A fundamental result states  that $s_\lambda(x_1,\ldots,x_r)$
is symmetric under permutation of $x_1, \ldots, x_r$. This has an elegant
proof by the Bender--Knuth involution: see for instance Theorem 7.10.2 in \cite{StanleyII}.

Let $|t|$ denote the sum of the entries of $t \in \SSYT_r(\lambda)$.
Specializing $s_\lambda$ by $x_k \mapsto q^k$ we
obtain 
%the generating function enumerating semistandard $\lambda$-tableaux by their sum
%of entries:
%Specializing $s_\lambda$ by $x_1 \mapsto q$, \ldots, $x_r \mapsto q^r$, we obtain
%the generating function enumerating semistandard $\lambda$-tableaux by their
%sum of entries. Thus, writing $|t|$ for the sum of the entries of $t$, we have
\begin{equation}\label{eq:qSpec}
 s_\lambda(q,\ldots, q^r) = \sum_{t \in \SSYT_r(\lambda)} q^{|t|}. \end{equation}
The minimum possible value of $|t|$ for $t \in \SSYT_r(\lambda)$ 
is $B(\lambda) = \sum_{i=1}^{\ell(\lambda)}i \lambda_i$.
%attained uniquely by the semistandard $\lambda$-tableau having $\lambda_i$ entries
%of $i$ in its $i$th row.
Stanley's Hook Content Formula may be stated as
\begin{equation}\label{eq:HCF}
 %\sum_{t \in \SSYT_r(\lambda)} q^{|t|} 
 s_\lambda(q,\ldots, q^r) = q^{B(\lambda)} \!\! \prod_{(i,j) \in [\lambda]}
\frac{q^{r+j-i}-1}{q^{h_{(i,j)}(\lambda) }-1} \end{equation}
where the hook lengths $h_{(i,j)}(\lambda)$ for $(i,j) \in [\lambda]$ are as we have defined.
(The term `content' refers to the powers of $q$ in the numerators.) Stanley's
formula was first proved in \cite[Theorem 15.3]{StanleyTheoryPlanePartitionsII}; 
for the statement above see \cite[Theorem 7.21.2]{StanleyII} and the following discussion.
For example, 
\[ s_{(3,2,1)}(q,q^2,q^3) = q^{10} + 2q^{11} + 2q^{12} + 2q^{13} + q^{14} \]
%\begin{align*}
%s_{(3,2,1)}(q,q^2,q^3) &= q^{10} \frac{(q-1)(q^2-1)(q^3-1)^2(q^4-1)(q^5-1)}{(q-1)^3(q^3-1)^2(q^5-1)} \\ &=
%q^{10} \frac{(q^2-1)(q^4-1)}{(q-1)^2} \\
%&= q^{10} (1+q)(1+q+q^2+q^3)  \\
%&= q^{10} + 2q^{11} + 2q^{12} + 2q^{13} + q^14 
%\end{align*}
enumerates the semistandard tableaux
\[ \young(111,22,3)\, \ \young(111,23,3)\, \ \young(112,22,3)\, \ \young(112,23,3)\, \
\young(113,22,3)\, \ \young(122,23,3)\, \ \young(113,23,3)\, \ \young(123,23,3). \]
The central symmetry about $x^{12}$ in the coefficients in this example is a
special case of the following basic and well-known lemma, left to the reader in
Exercise~7.75 in \cite{StanleyII}.
%(This lemma is well known: for example, it follows
%easily from Exercise~7.75 in \cite{StanleyII}. We include a proof using only
%the facts stated above.)

\begin{lemma}\label{lemma:centralSym}
Let $\lambda$ be a partition of $n$. Then
\[ s_\lambda(q,q^2, \ldots, q^r) = q^{(r+1)n} s_\lambda(q^{-1},q^{-2},\ldots, q^{-r}). \]
\end{lemma}

\begin{proof}
Let $f(x_1,x_2,\ldots, x_r)$ be a symmetric polynomial. If
$x_1^{e_1} \ldots  x_r^{e_r}$ is a monomial in $f$ then so is $x_1^{e_r} \ldots x_r^{e_1}$,
and the coefficients agree. Under the specialization $x_k \mapsto q^k$ the first
becomes $q^{e_1 + \cdots + re_r}$ and the second $q^{e_r + \cdots + re_1}$. Observe
that the sum of exponents is $(r+1)(e_1 + \cdots + e_r) = (r+1)n$. Therefore
the coefficients of $q^d$ and $q^{(r+1)n- d}$ in $f(q,q^2,\ldots,q^r)$ agree
for each $d$. Taking $f = s_\lambda$
this easily implies the lemma.
%If $g(y_1,y_2)$ is a two-variable symmetric polynomial of homogeneous degree $d$ then
%the coefficients in $g$ of $y_1^{d-e}y_2^e$ and $y_2^{d-e}y_1^e$ are equal. Therefore
%\[ g(y_1,y_2)= (y_1y_2)^d g(y_1^{-1},y_2^{-1}).\] 
%Since 
%the $r$-variable polynomial 
%$s_\lambda(x_1, \ldots, x_r)$ is symmetric, the $2$-variable
%polynomial $s_\lambda(y_1^{r-1}, y_1^{r-2}y_2,\ldots, y_2^{r-1})$ is also symmetric.
%It is homogeneous of degree $(r-1)n$.
%By the displayed equation immediately above, 
%\[ s_\lambda(y_1^{r-1}, y_1^{r-2}y_2, \ldots, y_2^{r-1} ) = (y_1y_2)^{(r-1)n} 
%s_\lambda(y_1^{-(r-1)},
%y_1^{-(r-2)}y_2^{-1}, \ldots, y_2^{-(r-1)}). \]
%Specializing by setting $y_1 = 1$ and $y_2 = q$ we obtain
%\[ s_\lambda(1,q, \ldots, q^{r-1}) = q^{(r-1)n} s_\lambda(1,q^{-1},\ldots,q^{-(r-1)}). \]
%Since $s_\lambda(y_1,\ldots,y_r)$ is homogeneous of degree $n$ we have
%$s_\lambda(1,q,\ldots, q^{r-1}) = q^{-n} s_\lambda(q,q^2, \ldots, q^r)$
%and $s_\lambda(1,q^{-1},\ldots, q^{-(r-1)}) = q^n s_\lambda(1, q^{-1}, \ldots, q^{-r})$.
%Substituting we get the lemma.
\end{proof}

The end of  our proof requires the following unique factorization theorem, implicitly
used in (4.8) in \cite{KingSU2Plethysms}.

\begin{lemma}\label{lemma:fullCyclotomic}
Let $\mathcal{E}$ and $\mathcal{E}'$ be finite multisubsets of $\N$. In the ring $\CC[q]$,
$\prod_{e \in \mathcal{E}} (q^e - 1) = \prod_{e' \in \mathcal{E}'} (q^{e'} - 1)$ if and only if $\mathcal{E} = \mathcal{E}'$.
\end{lemma}

\begin{proof}
If either $\mathcal{E}$ or $\mathcal{E}'$ is empty the result is obvious. In the remaining cases, 
let $u$ be maximal such that $\prod_{e \in \mathcal{E}} (q^e - 1)$ has $\mathrm{e}^{2\pi \mathrm{i}/u}$
%exp (2\pi \mathrm{i}/u)$ 
as a root. By maximality, $q^u-1$ is a factor in the left-hand side. Since $\mathrm{e}^{2\pi \mathrm{i}/u}$
is then also a root of $\prod_{e' \in \mathcal{E}'} (q^{e'} - 1)$ the same argument shows that $q^u-1$ is a factor in the right-hand side.
Therefore  $u = \max \mathcal{E} = \max \mathcal{E}'$.
It follows inductively by cancelling $q^u-1$ from both sides that $\mathcal{E} = \mathcal{E}'$.
\end{proof}

Let $\lambda^\circ$ denote the partition defined by deleting any final zeroes from
$(c-\lambda_{r}, \ldots,c-\lambda_1)$; here if $i > \ell(\lambda)$ we take $\lambda_i = 0$.

We require the  following bijection which is indicated in
\cite[\S 4]{KingSU2Plethysms}; we give a complete proof.

\begin{proposition}\label{prop:KingBijection}
There is a  bijection
\[ \SSYT_r(\lambda) \rightarrow \SSYT_r(\lambda^\circ) \]
defined by sending $t \in \SSYT_r(\lambda)$ to the unique
$\lambda^\circ$-tableau $t^\circ$ having
as its entries in column $j$ the complement in $\{1,\ldots, r\}$ of the entries of $t$ in column 
$c+1-j$, arranged in increasing order from north to south.
\end{proposition}

\begin{proof}
It suffices to prove that $t^\circ$ is semistandard.
Suppose, for a contradiction, that columns
$c-j-1$ and $c-j$ of 
$t^\circ$ have entries $\ell_1^\circ \le  k_1^\circ$, \ldots, $\ell_{i-1}^\circ \le k_{i-1}^\circ$ and 
$\ell_i^\circ > k_i^\circ$ read from north to south. 
Let columns $j$ and $j+1$ of $t$ read
from north to south have entries $k_1 \le \ell_1$, \ldots, $k_h \le \ell_h$ where
$h$ is maximal such that $\ell_h < \ell_i^\circ$.
%$k_h \le k_i^\circ$ and $k_h$ has an adjacent entry in column $j+1$.
Then $\{\ell_1^\circ, \ldots, \ell_{i-1}^\circ, \ell_1, \ldots, \ell_h\}$ are all the numbers strictly less than $\ell_i^\circ$ in $\{1,\ldots, r\}$, since, by choice of $h$, if $\ell_{h+1}$ is defined then $\ell_{h+1} > \ell_i^\circ$.
But from the chain $\ell_i^\circ > k_i^\circ > \ldots > k_1^\circ$ and the inequalities $\ell_i^\circ > \ell_h \ge \ell_j \ge k_j$
for $j \in \{1,\ldots, h\}$, we see that $\ell_i$ is strictly greater than $i+h$ distinct numbers, a contradiction.
\end{proof}

\subsection{Proof of Theorem~\ref{thm:main}}
Observe that if $t \in \SSYT_r(\lambda)$ then 
$|t| + |t^\circ| = r(1 + \cdots + c) = rc(c+1)/2$. Therefore
by~\eqref{eq:qSpec}, the bijection in Proposition~\ref{prop:KingBijection},
and then Lemma~\ref{lemma:centralSym}, we have
\begin{align*}
s_{\lambda^\circ}(q,\ldots, q^r) &= \sum_{u \in \SSYT_r(\lambda^\circ)} q^{|u|} \\
&= \sum_{t \in \SSYT_r(\lambda)} q^{(r+1)rc/2 - |t|}  \\
&= q^{r(r+1)c/2} s_\lambda(q^{-1},\ldots, q^{-r}) \\
&= q^{(r+1)(rc/2 -n)} s_\lambda(q,\ldots, q^r).
\end{align*}
Applying Stanley's Hook Content Formula~\eqref{eq:HCF} to each side we obtain
\[ q^{B(\lambda^\circ)}\!\! \prod_{(i,j) \in [\lambda^\circ]} \frac{q^{r + j - i} - 1}{q^{h_{(i,j)}(\lambda^\circ)} - 1} =
q^{(r+1)(rc/2 -n) + B(\lambda)}\!\!
\prod_{(i,j) \in [\lambda]} \frac{q^{r + j - i} - 1}{q^{h_{(i,j)}(\lambda)} - 1} .
  \]  
We now relate the numerators to the distances in Theorem~\ref{thm:main}, using
the bijection from  $[\lambda^\circ]$ to $D \backslash [\lambda]$
defined by $(i,j) \mapsto (i',j')$ where $i' = r+1-i$ and $j'=c+1-j$. 
%Note that $(i,j) \in [\lambda^\circ]$ 
%if and only if $(i',j') \in D\backslash [\lambda]$. 
We have
$h_{(i',j')}(\lambda) = h_{(i,j)}(\lambda^\circ)$. Moreover,
$d_{(i',j')}(\lambda)$ is the number of boxes in any walk by steps north and east 
from $(r+1-i,c+1-j)$ to $(1,c)$, namely $r-i + j$. Therefore the left-hand side
is 
\[ q^{B(\lambda^\circ)} \prod_{(i',j') \in D \backslash [\lambda]} 
\frac{q^{d_{(i',j')}(\lambda)}-1}{q^{h_{(i',j')}(\lambda)} - 1}.\]
If $(i,j) \in [\lambda]$ then $d_{(i,j)}(\lambda)$ is the number of
boxes in any walk by steps south
and west to~$(r,1)$, again $r-i + j$. Therefore,
cancelling the powers of $q$ we obtain
\[ \prod_{(i',j') \in D\backslash [\lambda]} 
\frac{q^{d_{(i',j')}(\lambda)}-1}{q^{h_{(i',j')}(\lambda)} - 1} = 
\prod_{(i,j) \in [\lambda]} \frac{q^{d_{(i,j)}(\lambda)} - 1}{q^{h_{(i,j)}(\lambda)} - 1}. \]
Theorem~\ref{thm:main} now follows by multiplying through by the denominators
and applying Lemma~\ref{lemma:fullCyclotomic}.

\def\cprime{$'$} \def\Dbar{\leavevmode\lower.6ex\hbox to 0pt{\hskip-.23ex
  \accent"16\hss}D} \def\cprime{$'$}
\providecommand{\bysame}{\leavevmode\hbox to3em{\hrulefill}\thinspace}
\providecommand{\MR}{\relax\ifhmode\unskip\space\fi MR }
\providecommand{\MRhref}[2]{%
  \href{http://www.ams.org/mathscinet-getitem?mr=#1}{#2}
}
\providecommand{\href}[2]{#2}

\end{document}